\documentclass[12pt]{amsart}
\usepackage{amsmath, amsthm, amssymb}
\usepackage[initials]{amsrefs}

\title{On VC-minimal fields and dp-smallness}
\author{Vincent Guingona}
\address{University of Notre Dame \\ Department of Mathematics \\ 255 Hurley, Notre Dame, IN 46556}
\email{guingona.1@nd.edu}
\date{\today}
\thanks{
        \indent The author was supported by NSF grant DMS-0838506.}

\newtheorem{thm}{Theorem}[section]
\newtheorem{cor}[thm]{Corollary}
\newtheorem{lem}[thm]{Lemma}
\newtheorem{prop}[thm]{Proposition}

\newtheorem{conj}[thm]{Conjecture}

\theoremstyle{remark}
\newtheorem{rem}[thm]{Remark}

\theoremstyle{definition}
\newtheorem{defn}[thm]{Definition}

\newcommand{\Th}{\mathrm{Th} }

\newcommand{\PP}{\mathrm{PP}}
\newcommand{\PPt}{\widetilde{\mathrm{PP}}}

\begin{document}

\begin{abstract}
 In this paper, we show that VC-minimal ordered fields are real closed.  We introduce a notion, strictly between convexly orderable and dp-minimal, that we call dp-small, and show that this is enough to characterize many algebraic theories.  For example, dp-small ordered groups are abelian divisible and dp-small ordered fields are real closed.
\end{abstract}

\maketitle

\section*{Introduction}

Recently, model theorists have been working on using the progress in stability theory as a template for work in unstable theories.  Since much of modern mathematics is done outside the stable world, it seems reasonable to explore such avenues.

The notion of a good ``minimality'' condition comes up frequently in stability theory, and there have been many useful suggestions for a suitable ``minimality'' condition in the unstable context.  S. Shelah developed dp-minimality, which was subsequently studied extensively by many others \cites{dgl, goodrick, ou, simon}.  Another property, strictly stronger than dp-minimality, that was extensively studied is weak o-minimality \cite{mms}.  In \cite{adl}, H. Adler introduces the notion of VC-minimality, which sits strictly between weak o-minimality and dp-minimality.  This too has been studied a great deal recently \cites{cs, fg, fg2, gl}.  In \cite{gl}, this author and M. C. Laskowksi develop a new ``minimality'' notion called ``convex orderability,'' which sits strictly between VC-minimality and dp-minimality.

When turned toward specific classes of theories, these minimality properties can yield strong classification results.  For example, Theorem 5.1 of \cite{mms} asserts that every weakly o-minimal ordered group is abelian divisible and Theorem 5.3 of \cite{mms} says that every weakly o-minimal ordered field is real closed.  For another example, Proposition 3.1 of \cite{simon} yields that every dp-minimal group is abelian by finite exponent and Proposition 3.3 of \cite{simon} states that every dp-minimal ordered group is abelian.  In a similar spirit, J. Flenner and this author show, in \cite{fg2}, that all convexly orderable ordered groups are abelian divisible.

The goal of this paper is two-fold.  In the first part of this paper, we introduce a new ``minimality'' condition that we call ``dp-smallness,'' which fits strictly between convex orderability and dp-minimality.  We then show that most of the results of \cite{fg2} work when we replace ``convexly orderable'' with ``dp-small.''  The second part of the paper is devoted to answering, in the affirmative, Open Question 3.7 of \cite{fg2}.  That is, we show that every convexly orderable ordered field is real closed.  Stronger than that, we actually show this for dp-small ordered fields.

\subsection*{Notation}

Throughout this paper, let $T$ be a complete theory in a language $L$ with monster model $\mathcal{U}$.  We will use $x$, $y$, $z$, etc. to stand for tuples of variables (instead of the cumbersome $\overline{x}$ or $\vec{x}$).  For any $A \subseteq \mathcal{U}$ and tuple $x$, let $A_x$ denote the set of all $|x|$-tuples from $A$ (so $A_x = A^{|x|}$).  If $|x| = 1$, we will say that $x$ is of the \emph{home sort}.\footnote{We could also consider theories with multiple sorts, but for the purposes of this paper, we will need a single ``home sort.''}  For a formula $\varphi(x)$ and $A \subseteq \mathcal{U}$, let
\[
 \varphi(A) = \{ a \in A_x : \mathcal{U} \models \varphi(a) \}.
\]
For ordered groups $G$, let $G_+$ denote the set of positive elements of $G$.  Similarly define $F_+$ for ordered fields $F$.  For a dense ordered group $G$, let $\overline{G}$ denote the completion of $G$ (in the sense of the ordering).  For valued fields $(F, v, \Gamma)$ (where $v : F^{\times} \rightarrow \Gamma$ is the valuation), for $a, b \in F$, let $a | b$ hold if and only if $v(a) \le v(b)$.  

\subsection*{Outline}

In Section \ref{Sec_DefRes}, we give all the relevant definitions and state the main results of the paper.  We define dp-smallness in Definition \ref{defn_WeakCO}.  Theorem \ref{thm_MainResults2} shows that many of the results of \cite{fg2} hold for dp-smallness instead of convex orderability.  Finally, Theorem \ref{Thm_CoRC} states that all dp-small ordered fields are real closed, generalizing Theorem 5.3 of \cite{mms}.  In Section \ref{Sec_WCO}, we provide a proof that dp-smallness does fit strictly between convex orderability and dp-minimality.  We also prove Theorem \ref{thm_MainResults2}.  In Section \ref{Sec_VCMinOF}, we prove Theorem \ref{Thm_CoRC}.  In Section \ref{Sec_VCMinF}, we discuss VC-minimal fields in general.  We show that VC-minimal stable fields are algebraically closed and conjecture that all VC-minimal fields are either algebraically closed or real closed.

\section{Definitions and Results}\label{Sec_DefRes}

The following definition is due to H. Adler in \cite{adl}.

\begin{defn}\label{Defn_VCMin}
 Fix a set of formulas $\Psi = \{ \psi_i(x; y_i) : i \in I \}$ (where $x$ is a free variable in every formula, but the $y_i$'s may vary).  We say that $\Psi$ is \emph{directed} if, for all $i, j \in I$, $b \in \mathcal{U}_{y_i}$, and $c \in \mathcal{U}_{y_j}$, we have that one of the following holds:
 \begin{enumerate}
  \item $\models \forall x ( \psi_i(x; b) \rightarrow \psi_j(x; c) )$,
  \item $\models \forall x ( \psi_j(x; c) \rightarrow \psi_i(x; b) )$, or
  \item $\models \neg \exists x ( \psi_i(x; b) \wedge \psi_j(x; c) )$.
 \end{enumerate}
 We say that a theory $T$ is \emph{VC-minimal} if there exists a directed $\Psi = \{ \psi_i(x; y_i) : i \in I \}$ where $x$ is of the home sort and each $L(\mathcal{U})$-formula $\theta(x)$ is a boolean combination of instances of formulas from $\Psi$.
\end{defn}

This is an important concept because it generalizes some other ``minimal'' notions in model theory.  For example, every strongly minimal theory is VC-minimal and every o-minimal theory is VC-minimal.  Moreover, a few interesting algebraic examples are VC-minimal, including algebraically closed valued fields.  In \cite{fg2}, J. Flenner and this author classify VC-minimality in certain algebraic structures using an intermediate tool called convex orderability.  This notion was first introduced in \cite{gl}.

\begin{defn}\label{Defn_CO}
 An $L$-structure $M$ is \emph{convexly orderable} if there exists $\lhd$ a linear order on $M$ (not necessarily definable) such that, for all $L$-formulas $\varphi(x; y)$ with $x$ in the home sort, there exists $K_\varphi < \omega$ such that, for all $b \in M_y$, the set $\varphi(M; b)$ is a union of at most $K_\varphi$ $\lhd$-convex subsets of $M$.
\end{defn}

It is shown in \cite{gl} (Proposition 2.3) that convex orderability is an elementary property, so we say a theory $T$ is \emph{convexly orderable} if for any (equivalently all) $M \models T$, $M$ is convexly orderable.  It is also shown in \cite{gl} (Theorem 2.4) that all VC-minimal theories are convexly orderable.  One reason convex orderability is preferable over VC-minimality is that it is easier to show a theory is not convexly orderable.  This is the main tool used in obtaining the results from \cite{fg2}.

\begin{thm}[Main results of \cite{fg2}]\label{thm_MainResults}
 The following hold:
 \begin{enumerate}
  \item If $T = \Th(G; +, <)$ is the theory of an infinite ordered group, then $T$ is convexly orderable if and only if $G$ is abelian divisible.
  \item If $T = \Th(F; +, \cdot, <)$ is the theory of a convexly orderable ordered field, then all positive elements of $F$ have an $n$th root for all $n < \omega$.
  \item If $T = \Th(A; +)$ is the theory of an abelian group, then $T$ is convexly orderable if and only if $T$ is dp-minimal and $A$ has upward coherence (see Definition \ref{defn_UpwardCoherence} below).
  \item If $T = \Th(F; +, \cdot, |)$ is the theory of an Henselian valued field that is convexly orderable, then $\Gamma$ the value group is divisible.
 \end{enumerate}
\end{thm}

In each of the cases above, the crux of the argument is using a combinatorial property about the theory to show that it can not be convexly ordered.  It boils down to the following notion.

\begin{defn}\label{defn_WeakCO}
 We say that a partial type $\pi(x)$ is \emph{dp-small} if there does not exist $\varphi_i(x)$ an $L(\mathcal{U})$-formula for $i < \omega$, $\psi(x; y)$ an $L$-formula, and $b_j \in \mathcal{U}_y$ for $j < \omega$ such that, for all $i_0, j_0 < \omega$, the type
 \[
  \pi(x) \cup \{ \varphi_{i_0}(x), \psi(x; b_{j_0}) \} \cup \{ \neg \varphi_i(x) : i \neq i_0 \} \cup \{ \neg \psi(x; b_j) : j \neq j_0 \}
 \]
 is consistent.  We say $T$ is \emph{dp-small} if $x=x$ is dp-small where $x$ is of the home sort.
\end{defn}

Compare this to the definition of ICT-patterns and dp-minimality (Definition \ref{Defn_dpMin} below).  In fact, dp-smallness implies dp-minimality (see Proposition \ref{Prop_wCOdpMin} below).

As promised, we have the following relationship between dp-small and convexly orderable.

\begin{prop}\label{prop_wCOCO}
 If $T$ is convexly orderable, then $T$ is dp-small.
\end{prop}

In this paper, we show that dp-smallness is enough to get all the results in Theorem \ref{thm_MainResults}.  This generalizes most of the results from \cite{fg2}.  That is,

\begin{thm}[Results of \cite{fg2}, revisited]\label{thm_MainResults2}
 The following hold:
 \begin{enumerate}
  \item If $T = \Th(G; +, <)$ is the theory of an infinite ordered group, then $T$ is dp-small if and only if $G$ is abelian divisible.
  \item If $T = \Th(F; +, \cdot, <)$ is the theory of a dp-small ordered field, then all positive elements of $F$ have an $n$th root for all $n < \omega$.
  \item If $T = \Th(A; +)$ is the theory of an abelian group, then $T$ is dp-small if and only if $T$ is dp-minimal and $A$ has upward coherence.
  \item If $T = \Th(F; +, \cdot, |)$ is the theory of a dp-small Henselian valued field, then $\Gamma$ the value group is divisible.
 \end{enumerate}
\end{thm}

Beyond this, the other main result of this paper is the following.  This answers Open Question 3.7 of \cite{fg2} in the affirmative.

\begin{thm}\label{Thm_CoRC}
 Suppose that $\mathfrak{F} = (F; +, \cdot, <)$ is an ordered field and $T = \Th(\mathfrak{F})$.  The following are equivalent.
 \begin{enumerate}
  \item $T$ is VC-minimal.
  \item $T$ is convexly orderable.
  \item $T$ is dp-small.
  \item $\mathfrak{F}$ is real closed.
 \end{enumerate}
\end{thm}

As an immediate corollary, we get the following.

\begin{cor}\label{Cor_CoVF}
 Suppose that $\mathfrak{F} = (F; +, \cdot, <, |)$ is a dp-small ordered Henselian valued field (with non-trivial valuation).  Then $\mathfrak{F}$ is a real closed valued field.
\end{cor}

\section{dp-Smallness}\label{Sec_WCO}

In this section, we prove Proposition \ref{prop_wCOCO} and Theorem \ref{thm_MainResults2}.

\begin{proof}[Proof of Proposition \ref{prop_wCOCO}]
 Suppose $T$ is not dp-small.  Therefore, there exists $L(\mathcal{U})$-formulas $\varphi_i(x)$ for $i < \omega$, an $L$-formula $\psi(x; y)$, and $b_j \in \mathcal{U}_y$ for $j < \omega$ ($x$ is of the home sort) such that, for all $i_0, j_0 < \omega$,
 \[
  \{ \varphi_{i_0}(x), \psi(x; b_{j_0}) \} \cup \{ \neg \varphi_i(x) : i \neq i_0 \} \cup \{ \neg \psi(x; b_j) : j \neq j_0 \}
 \]
 is consistent.  By replacing $\varphi_i(x)$ with $\varphi'_i(x) = \varphi_i(x) \wedge \bigwedge_{i' < i} \neg \varphi_{i'}(x)$, we may assume that the $\varphi_i(x)$ are pairwise inconsistent.
 
 By means of contradiction, suppose $T$ is convexly orderable.  Say $\lhd$ is a convex ordering on $\mathcal{U}$.  Further, let $K < \omega$ be such that, for all $b \in \mathcal{U}_y$, $\psi(\mathcal{U}; b)$ is a union of at most $K$ $\lhd$-convex subsets of $\mathcal{U}$.  Now look at $\varphi_i(x)$ for $i \le 2K$ and suppose $L < \omega$ is such that, for all $i \le 2K$, $\varphi_i(\mathcal{U})$ is a union of at most $L$ $\lhd$-convex subsets of $\mathcal{U}$.  Let $C_{i,\ell}$ enumerate these.  That is, $C_{i,\ell} \subseteq \mathcal{U}$ is $\lhd$-convex and, for each $i \le 2K$,
 \[
  \varphi_i(\mathcal{U}) = \bigcup_{\ell < L} C_{i,\ell}.
 \]
 By definition (and saturation of $\mathcal{U}$), for each $i \le 2K$ and $j < \omega$,
 \[
  \varphi_{i}(\mathcal{U}) \cap \psi(\mathcal{U}; b_{j}) \setminus \left( \bigcup_{j' \neq j} \psi(\mathcal{U}; b_{j'}) \right) \neq \emptyset.
 \]
 By pigeon-hole, there exists $J \subseteq \omega$ infinite such that, for each $i \le 2K$, there exists $\ell_i < L$ such that, for all $j \in J$,
 \[
  C_{i, \ell_i} \cap \psi(\mathcal{U}; b_{j}) \setminus \left( \bigcup_{j' \neq j} \psi(\mathcal{U}; b_{j'}) \right) \neq \emptyset.
 \]
 In particular, for any fixed $j \in J$, for all $i \le 2K$,
 \[
  C_{i, \ell_i} \cap \psi(\mathcal{U}; b_j) \neq \emptyset \text{ and } C_{i, \ell_i} \cap \neg \psi(\mathcal{U}; b_j) \neq \emptyset.
 \]
 Without loss of generality, suppose $C_{0,\ell_0} \lhd C_{1, \ell_1} \lhd ... \lhd C_{2K, \ell_{2K}}$.  For each even $i \le 2K$, choose $a_i \in C_{i, \ell_i} \cap \psi(\mathcal{U}; b_j)$ and, for each odd $i \le 2K$, choose $a_i \in C_{i, \ell_i} \cap \neg \psi(\mathcal{U}; b_j)$.  Thus, $a_0 \lhd ... \lhd a_{2K}$ but it alternates belonging to $\psi(\mathcal{U}; b_j)$.  This contradicts the fact that $\psi(\mathcal{U}; b_j)$ is a union of at most $K$ $\lhd$-convex subsets of $\mathcal{U}$.
\end{proof}

To see that dp-smallness is, in fact, distinct from convex orderability, consider an example from \cite{dgl}.  Let $L$ be the language consisting of unary predicates $P_i$ for $i \in \omega_1$.  Let $T$ be the theory stating that, for each finite $I, J \subseteq \omega_1$ with $I \cap J = \emptyset$, there are infinitely many elements realizing
\[
 \bigwedge_{i \in I} P_i(x) \wedge \bigwedge_{j \in J} \neg P_j(x).
\]
By Proposition 3.6 of \cite{dgl}, this theory is complete, has quantifier elimination, but is not VC-minimal.  By Example 2.10 of \cite{gl}, $T$ is not convexly orderable.  However, $T$ is dp-small.  To see this, notice that any supposed witness to non-dp-smallness would involve only countably many predicates $P_i$, but the reduct to countably many predicates is VC-minimal (see the discussion after Example 2.10 of \cite{gl}).  How does dp-small compare to dp-minimality?

\begin{defn}\label{Defn_dpMin}
 A partial type $\pi(x)$ is \emph{dp-minimal} if there does not exist $L$-formulas $\varphi(x; y)$ and $\psi(x; z)$, $a_i \in \mathcal{U}_y$ for $i < \omega$, and $b_j \in \mathcal{U}_z$ for $j < \omega$ such that, for all $i_0, j_0 < \omega$, the type
 \[
  \pi(x) \cup \{ \varphi(x; a_{i_0}), \psi(x; b_{j_0}) \} \cup \{ \neg \varphi(x; a_i) : i \neq i_0 \} \cup \{ \neg \psi(x; b_j) : j \neq j_0 \}
 \]
 is consistent.  We call such a witness to non-dp-minimality an \emph{ICT-pattern}.  We say $T$ is \emph{dp-minimal} if $x=x$ is dp-minimal where $x$ is of the home sort.
\end{defn}

\begin{prop}\label{Prop_wCOdpMin}
 If a partial type $\pi(x)$ is dp-small, then $\pi$ is dp-minimal.  In particular, all dp-small theories are dp-minimal.
\end{prop}

\begin{proof}
 Notice that an ICT-pattern is, in particular, a witness to non-dp-smallness (where all the $\varphi_i(x)$ happen to be $\varphi(x; a_i)$).
\end{proof}

Notice that Theorem \ref{thm_MainResults2} gives us examples of theories which are not dp-small but are dp-minimal.  For example, the theory of Presburger arithmetic, $\Th(\mathbb{Z}; +, <)$, and the theory of the $p$-adics, $\Th(\mathbb{Q}_p; +, \cdot, |)$.

So we have the following picture, where each implication is strict:
\[
 \text{VC-minimal} \Rightarrow \text{convexly orderable} \Rightarrow \\ \text{dp-small} \Rightarrow \text{dp-minimal}.
\]

We are now ready to prove Theorem \ref{thm_MainResults2}.  This basically involves minor tweaks to the proofs presented in \cite{fg2}, so we will skip some details here.  First, we tackle the theory of ordered groups.

\begin{proof}[Proof of Theorem \ref{thm_MainResults2} (1)]
 Let $T = \Th(G; \cdot, <)$ the theory of an ordered group.  If $G$ is abelian and divisible, then $G$ is o-minimal, hence VC-minimal, hence convexly orderable, hence dp-small.  Conversely, suppose $G$ is dp-small.  By Proposition \ref{Prop_wCOdpMin}, it is dp-minimal, hence by Proposition 3.3 of \cite{simon}, $G$ is abelian.  If $G$ is not divisible, suppose we have a prime $p$ so that $pG \neq G$.  Check that the formulas
 \[
  \varphi_i(x) = ( p^i | x ) \wedge ( p^{i+1} \not | x )
 \]
 and $\psi(x; y, z) = y < x < z$ witness that $T$ is not dp-small.  This amounts to showing that $\varphi_i(\mathcal{U})$ is cofinal in $G$, which is Lemma 3.3 of \cite{fg2}.
\end{proof}

Notice that Theorem \ref{thm_MainResults2} (2) follows as an immediate corollary, since $(F_+; \cdot, <)$ is a dp-small ordered group, hence is divisible by Theorem \ref{thm_MainResults2} (1).  We turn our attention to abelian groups.  First, we recall some definitions from \cite{fg2}.

Let $T = \Th(A; +)$ for $A$ some abelian group.  For definable subgroups $B_0, B_1 \subseteq A$, define $\precsim$ a quasi-ordering as follows:
\[
 B_0 \precsim B_1 \text{ if and only if } [ B_0 : B_0 \cap B_1 ] < \aleph_0.
\]
This generates an equivalence relation $\sim$.  Let $\PP(A)$ be the set of all p.p.-definable subgroups of $A$, which are finite intersections of subgroups of the form $\varphi_{k,m}(A)$ where
\[
 \varphi_{k,m}(x) = \exists y ( k \cdot y = m \cdot x ).
\]
Let $\PPt(A) = \PP(A) / \sim$.  Notice that $(\PPt(A); \precsim)$ is a partial order.

\begin{prop}[Corollary 4.12 of \cite{ADHMS2}]\label{prop_dpMinAG}
 The theory $T$ is dp-minimal if and only if $(\PPt(A); \precsim)$ is a linear order.
\end{prop}

\begin{defn}[Definition 5.9 of \cite{fg2}]\label{defn_UpwardCoherence}
 For $X \in \PPt(A)$, we say $X$ is \emph{upwardly coherent} if there exists $B \in X$ such that, for all $B_1 \in \PP(A)$ with $B \precnsim B_1$, $B \subseteq B_1$.  We say the group $A$ is \emph{upwardly coherent} if every $X \in \PPt(A)$ is upwardly coherent.
\end{defn}

We are now ready to prove the next part of Theorem \ref{thm_MainResults2}.

\begin{proof}[Proof of Theorem \ref{thm_MainResults2} (3)]
 If $T$ is dp-minimal and $A$ is upwardly coherent, then by Theorem 5.11 of \cite{fg2}, $T$ is convexly orderable, hence dp-small.  Conversely, if $T$ is dp-small, then $T$ is dp-minimal by Proposition \ref{Prop_wCOdpMin}.  So suppose that $T$ is dp-minimal but $A$ is not upwardly coherent.  So fix $X \in \PPt(A)$ without upward coherence.  Follow the construction in the proof of Theorem 5.11 of \cite{fg2}.  This gives us subgroups $B \in X$ and
 \[
  A = A_0 \supseteq A_1 \supseteq A_2 \supseteq ... \text{ from } \PP(A)
 \]
 such that
 \begin{enumerate}
  \item for each $ i < \omega$, $A_i \cap B \neq A_{i+1} \cap B$ and
  \item for each $ i < \omega$, $[A_i : A_i \cap B ] \ge \aleph_0$.
 \end{enumerate}
 Now it is easy to see that $\varphi_i(x) = [ x \in (A_i \setminus A_{i+1}) ]$ and $\psi(x; y) = [ (x - y) \in B ]$ are a witness to non-dp-smallness.  For more details, see the proof of Lemma 5.7 of \cite{fg2}.
\end{proof}

Finally, we turn our attention to Theorem \ref{thm_MainResults2} (4).  However, as in \cite{fg2}, we will prove a more general result about simple interpretations.

\begin{defn}\label{defn_SimpleInterpret}
 Suppose $M$ and $N$ are structures in different languages, $A \subseteq M$, $S \subseteq M$ is $A$-definable, and $E \subseteq M \times M$ is an $A$-definable equivalence relation on $S$.  We say that $M$ \emph{simply interprets} $N$ over $A$ if the elements of $N$ are in bijection with $S / E$ and the relations on $S$ induced by the relations and functions on $N$ via this bijection are $A$-definable in $M$.
\end{defn}

\begin{lem}\label{Lem_SIwCO}
 If $M$ simply interprets $N$ over $\emptyset$ and $M$ is dp-small, then $N$ is dp-small.
\end{lem}

\begin{proof}
 Let $\sigma : N \rightarrow S$ be the map given by simple interpretability (so $\sigma$ induces a bijection from $N$ to $S/E$).  Suppose there exists $\varphi_i(x; z_i)$ for $i < \omega$ and $\psi(x;y)$ in the language of $N$, $a_i \in N_{z_i}$ for $i < \omega$, and $b_j \in N_y$ for $j < \omega$ witnessing non-dp-smallness.  By definition, there exists $\varphi^*_i(x; z_i)$ for $i < \omega$ and $\psi^*(x;y)$ in the language of $M$ corresponding to $\varphi_i(x; z_i)$ and $\psi(x; y)$ respectively.  Then, one checks that $\varphi^*_i(x; \sigma(a_i))$ and $\psi^*(x; \sigma(b_j))$ witness to the fact that $M$ is not dp-small.
\end{proof}

\begin{proof}[Proof of Theorem \ref{thm_MainResults2} (4)]
 Let $T = \Th(K; +, \cdot, |)$ a Henselian valued field.  Check that the value group and residue field are simply interpretable in $K$.  Then, by Lemma \ref{Lem_SIwCO} and Theorem \ref{thm_MainResults2} (1), the value group is divisible.
\end{proof}

Moreover, when we include the ordering in the language, by Lemma \ref{Lem_SIwCO} and Theorem \ref{Thm_CoRC}, the residue field is real closed.  This, along with Theorem \ref{thm_MainResults2} (4), gives us Corollary \ref{Cor_CoVF}.  So, with this in mind, we switch gears and prove Theorem \ref{Thm_CoRC}.

\section{VC-Minimal Ordered Fields}\label{Sec_VCMinOF}

In this section, we prove Theorem \ref{Thm_CoRC}.  In order to do this, we follow the proof of Theorem 5.3 in \cite{mms}.  First, we give a trivial consequence of Theorem 3.6 of \cite{simon}, but this formulation is useful to us here.

\begin{lem}\label{Lem_PierresLemma}
 If $(G; <, +, ...)$ is a dp-minimal expansion of a divisible ordered abelian group and $X \subseteq G$ is definable, then $X$ is the union of finitely many points and an open set.
\end{lem}

\begin{proof}
 Let $\mathrm{Int}(X) = \{ a \in X : (\exists b, c \in X)[b < a < c \wedge (\forall x \in X)(b < x < c \rightarrow x \in X)] \}$ and let $\mathrm{Ext}(X) = X \setminus \mathrm{Int}(X)$.  Then, since $\mathrm{Ext}(X)$ is a definable subset of a dp-minimal expansion of a divisible ordered abelian group with no interior, by Theorem 3.6 of \cite{simon}, $\mathrm{Ext}(X)$ is finite.  This gives the desired conclusion.
\end{proof}

Next, we will need a result from \cite{goodrick}.

\begin{lem}[Lemma 3.19 of \cite{goodrick}]\label{Lem_Lem319}
 Let $\mathfrak{G} = (G; +, <, ...)$ be a dp-minimal expansion of a divisible ordered abelian group.  Let $f : G \rightarrow \overline{G}_+$ be a definable function (to $\overline{G}_+$ the positive elements of the completion of $G$).  Then, for each open interval $I \subseteq G$, there exists an open interval $J \subseteq I$ and $\epsilon > 0$ such that, for all $a \in J$, $f(a) > \epsilon$.
\end{lem}

The next ingredient is a slight modification of Theorem 3.2 of \cite{goodrick}.  We need to be very careful, because this theorem as stated only works for definable functions to $G$ and not necessarily to $\overline{G}$.  However, a simple modification of the proof of Theorem 3.4 of \cite{mms} gives us the desired result.

\begin{lem}\label{Lem_LocalMonotone}
 Let $\mathfrak{G} = (G; +, <, ...)$ be a dp-minimal expansion of a divisible ordered abelian group and let $f : G \rightarrow \overline{G}$ be a definable function.  Then, for any $b \in G$, there exists $\delta > 0$ such that, on the interval $(b, b+\delta)$, the function $f$ is monotone.
\end{lem}

\begin{proof}
 By Lemma \ref{Lem_PierresLemma}, for each $x \in G$, one of the following holds:
 \begin{enumerate}
  \item $\varphi_0(x) = (\exists x_1 > x)(\forall y)[x < y < x_1 \rightarrow f(x) < f(y)]$,
  \item $\varphi_1(x) = (\exists x_1 > x)(\forall y)[x < y < x_1 \rightarrow f(x) = f(y)]$, or
  \item $\varphi_2(x) = (\exists x_1 > x)(\forall y)[x < y < x_1 \rightarrow f(x) > f(y)]$.
 \end{enumerate}
 Again by Lemma \ref{Lem_PierresLemma}, there exists $\delta_1 > 0$ and $i=0,1,2$ such that, for all $a \in (b, b + \delta_1)$, $\models \varphi_i(a)$.  If $i = 1$, we are done.  Without loss of generality, suppose $i=0$ holds.  Let
 \[
  \chi(x) = (\forall x_1 > x)(\exists y, z)[x < y < z < x_1 \wedge f(y) \ge f(z)].
 \]
 We show that $\chi((b, b+\delta_1))$ is finite.  By Lemma \ref{Lem_PierresLemma}, it suffices to show that $\chi((b, b+\delta_1))$ has no interior.  So, suppose there exists an open interval $J \subseteq \chi((b, b+\delta_1))$.  Define a function $g : J \rightarrow \overline{G}_+$,
 \[
  g(x) = \sup \{ z - x : z \in J, x < z, (\forall y \in (x, z))(f(x) < f(y)) \}.
 \]
 Since $\models \varphi_0(a)$ for all $a \in J$, $g(a) > 0$.  Fix $\epsilon > 0$, $a, c \in J$ with $a < c$ and $c - a < \epsilon$.  Since $\chi(a)$, there exists $d \in (a,b) \subseteq J$ such that $f(a) \ge f(d)$.  Hence, $g(a) < \epsilon$.  This contradicts Lemma \ref{Lem_Lem319}.
 
 Therefore, $\chi((b, b+\delta_1))$ is finite, hence there exists $\delta > 0$ such that $f$ is strictly increasing on $(b, b+\delta)$.  The proof works similarly when $i = 2$ (and we get $f$ locally strictly decreasing after $b$).
\end{proof}

The following theorem is a simple generalization of Lemma \ref{Lem_PierresLemma} (Theorem 3.6 of \cite{simon}) and Lemma \ref{Lem_Lem319} (Lemma 3.19 of \cite{goodrick}) respectively.

\begin{thm}\label{Thm_GeneralizeDPMinSets}
 Let $\mathfrak{G} = (G; +, <, ...)$ be a dp-minimal expansion of a divisible ordered abelian group.  Then, for each $n \ge 1$,
 \begin{itemize}
  \item [ (1)${}_n$ ] If $X \subseteq G^n$ is a definable set with non-empty interior and $X = X_1 \cup ... \cup X_r$ is a definable partition of $X$, then, for some $i$, $X_i$ has non-empty interior.
  \item [ (2)${}_n$ ] If $f : G^n \rightarrow \overline{G}_+$ is a definable function, then for each open box $B \subseteq G^n$ there exists an open box $B' \subseteq B$ and $\epsilon > 0$ so that, for all $x \in B'$, $f(x) > \epsilon$.
 \end{itemize}
\end{thm}

In particular, this holds for all dp-small expansions of an ordered group by Theorem \ref{thm_MainResults2} (1).  This proof is easily adapted from the proof of Theorem 4.2 and Theorem 4.3 in \cite{mms}.

\begin{proof}[Proof of Theorem \ref{Thm_GeneralizeDPMinSets}]
 By simultaneous induction on $n$.  Note that (1)${}_1$ holds by Lemma \ref{Lem_PierresLemma} and (2)${}_1$ holds by Lemma \ref{Lem_Lem319}.  Suppose (1)${}_{n-1}$ and (2)${}_{n-1}$ holds.
 
 (1)${}_n$:  Let $X \subseteq G^n$ be a definable set with non-empty interior and $X = X_1 \cup ... \cup X_r$ be a definable partition of $X$.  Fix $B \subseteq X$ an open box and let $\pi : G^n \rightarrow G^{n-1}$ be the projection onto the first $n-1$ coordinates.  Choose any $a \in G$ so that $(b, a) \in B$ for some $b \in G^{n-1}$.  For each $i$, let
 \[
  Z_i = \{ b \in \pi(B) : (\exists a')(\forall x)(a < x < a' \rightarrow \langle b,x \rangle \in X_i) \}.
 \]
 We claim that $Z_i$ is a partition of $\pi(B)$.  To see this, take $b \in \pi(B)$ and consider
 \[
  Y_i = X_i |_b \cap (a, \infty) = \{ x > a : \langle b, x \rangle \in X_i \}.
 \]
 Since the $Y_i$ form a partition of $X |_b \cap (a, \infty)$ which contains $B |_b$, by Lemma \ref{Lem_PierresLemma}, there exists $i$ and $a' > a$ so that $\{ (b,x) : a < x < a' \} \subseteq X_i$.  Hence $b \in Z_i$.
 
 By (1)${}_{n-1}$ on the $Z_i$, there exists an open box $B^* \subseteq \pi(B)$ and $i$ so that $B^* \subseteq Z_i$.  By replacing $B$ with $(B^* \times G) \cap B$, we may assume that $\pi(B) \subseteq Z_i$.  For each $b \in \pi(B)$, let $f(b)$ be the supremum of $a' - a$ for all $a'$ as in the definition of $Z_i$.  If $b \in (G^{n-1} \setminus \pi(B))$, set $f(b) = \infty$.  Thus, clearly $f(b) > 0$ for all $b \in G^{n-1}$.  By (2)${}_{n-1}$, there exists $B' \subseteq \pi(B)$ and $\epsilon > 0$ so that, for all $x \in B'$, $f(x) > \epsilon$.  Therefore, $B' \times (a, a+\epsilon) \subseteq X_i$ hence $X_i$ has non-empty interior.
 
 (2)${}_n$: Let $f : G^n \rightarrow \overline{G}$ be a definable function and fix an open box $B \subseteq G^n$.  For each $b = \langle b_1, ..., b_n \rangle \in B$ and each $j = 1, ..., n$, let $B_j = \{ a \in G : \langle b_1, ..., b_{j-1}, a, b_{j+1}, ..., b_n \rangle \in B \}$, define a function $g_{b,j} : B_j \rightarrow \overline{G}$ as follows:
 \[
  g_{b,j}(a) = f(b_1, ..., b_{j-1}, a, b_{j+1}, ..., b_n).
 \]
 By Lemma \ref{Lem_LocalMonotone}, there exists $\delta \in G_+$ so that, on $(b_j, b_j + \delta)$, the function $g_{b,j}$ is monotonic (either strictly increasing, strictly decreasing, or constant).  As in the proof of Theorem 4.3 of \cite{mms}, we first use (1)${}_n$ on the type of monotonicity of $g_{b,j}$ after $b_j$ for each $j$ to reduce to a closed box $B'$ such that, for each $j$, either
 \begin{itemize}
  \item $g_{b,j}$ is strictly increasing locally after $b_j$ for all $b \in B'$,
  \item $g_{b,j}$ is strictly decreasing locally after $b_j$ for all $b \in B'$, or
  \item $g_{b,j}$ is constant locally after $b_j$ for all $b \in B'$.
 \end{itemize}
 Fix $c = \langle c_1, ..., c_n \rangle \in B'$.  For each $j$, let $\pi_j$ be the projection of $G^n$ onto $G^{n-1}$ by removing the $j$th coordinate, define a function $F_j : \pi_j(B') \rightarrow \overline{G}$ by setting $F_j(b_1, ..., b_{j-1}, b_{j+1}, ..., b_n)$ as the supremum of all $\delta \in G_+$ such that $g_{\langle b_1, ... b_{j-1}, c_j, b_{j+1}, ..., b_n\rangle,j}$ is monotone on the interval $(c_j, c_j + \delta)$.  Hence, $F_j(b) > 0$ for all $b \in \pi(B')$.  As in the proof of Theorem 4.3 of \cite{mms}, use (2)${}_{n-1}$ to shrink $B'$ to an open box such that, for all $j$ and $b \in B'$, $g_{b,j}$ is monotone on the projection of $B'$ to the $j$th coordinate (the smallest corner is $c$).
 
 Let $B' = I_1 \times ... \times I_n$.  For each $j$, if $g_{b,j}$ is non-decreasing on $I_j$ (for any equivalently all $b \in B'$), set $k_j$ to be the left endpoint of $I_j$ (i.e., $k_j = c_j$).  Otherwise, set $k_j$ to be the right endpoint of $I_j$.  Notice that, for all $b \in B'$, $f(b) \ge f(k_1, ..., k_n) > 0$.  This gives us the desired conclusion.
\end{proof}

\begin{cor}\label{Cor_BetweenTwoFunctions}
 Let $\mathfrak{G} = (G; +, <, ...)$ be a dp-minimal expansion of a divisible ordered abelian group and fix $n \ge 1$.  Suppose that $B$ is an open box and $f, h : B \rightarrow \overline{G}$ are definable functions such that $h(x) < f(x)$ for all $x \in B$.  If $f$ or $h$ is continuous, then $\{ \langle x, y \rangle : x \in B, h(x) < y < f(x) \}$ has non-empty interior.
\end{cor}

\begin{proof}
 Define $g : G^n \rightarrow G$ as follows
 \[
  g(x) =
  \begin{cases}
   f(x) - h(x) & \text{ if } x \in B, \\
   \infty & \text{ if } x \notin B.
  \end{cases}
 \]
 Since $g(x) > 0$ for all $x \in G^n$, by Theorem \ref{Thm_GeneralizeDPMinSets} (2)${}_n$, there exists an open box $B' \subseteq B$ and $\epsilon > 0$ so that, for all $x \in B'$, $g(x) > \epsilon$.  Since the argument is symmetric, assume $h$ is continuous.  Choose any $a \in B'$.  Since $h$ is continuous, there exists an open box $B'' \subseteq B'$ containing $a$ such that, for all $x \in B''$, $|h(x) - h(a)| < \epsilon / 3$.  It is easy to verify that $B'' \times (h(a) + \epsilon / 3, h(a) + 2 \epsilon / 3)$ is an open box contained in $\{ \langle x, y \rangle : x \in B, h(x) < y < f(x) \}$, as desired.
\end{proof}

Following the outline of \cite{mms}, we prove an analog of their Proposition 5.4 for dp-minimal ordered fields.

\begin{prop}\label{Prop_54dpmin}
 Let $F$ be a dp-minimal ordered field with real closure $R$, fix $\alpha \in R$, and suppose that for each $\epsilon \in R$ with $\epsilon > 0$, there exists $b \in F$ such that $|\alpha - b| < \epsilon$.  Then, $\alpha \in F$.
\end{prop}

This follows from the analog of Proposition 5.9 of \cite{mms} for dp-minimal ordered fields.

\begin{prop}\label{Prop_59dpmin}
 Fix $F$ a dp-minimal ordered field.  Let $p = \langle p_1, ..., p_n \rangle$ be an element of $F[x_1, ..., x_n]^n$ and $a \in F^n$ with $J_p(a) \neq 0$ (the Jacobian of $p$ at $a$).  Then, for every box $U \subseteq F^n$ containing $a$, the set $p(U)$ has non-empty interior.
\end{prop}

In order to prove this proposition, we use Lemma 5.5 and Corollary 5.8 of \cite{mms}.  These are true of any ordered field $F$.  We summarize these two in the following lemma.

\begin{lem}\label{Lem_58dpmin}
 Fix $F$ any ordered field and let $R$ be its real closure.  Let $p \in F[x_1, ..., x_n]^n$ and let $B \subseteq R^n$ be any open box whose endpoints lie in $F$.  Suppose that, for some $a \in (B \cap F^n)$, $J_p(a) \neq 0$.  Then, there is an open box $U \subseteq B$ whose endpoint lie in $F$ with $a \in U$ such that $p |_U$ is injective, $V := p(U)$ is open, and $p^{-1} |_{V \cap F^n}$ is continuous.
\end{lem}

\begin{proof}[Proof of Proposition \ref{Prop_59dpmin}]
 By induction on $n$.  For $n=1$, this follows from Lemma \ref{Lem_PierresLemma} (Theorem 3.6 of \cite{simon}).

 Let $p \in F[x]^n$ and $a = \langle a_1, ..., a_n \rangle \in F^n$ with $J_p(a) \neq 0$ and fix $U = I_1 \times ... \times I_n \subseteq F^n$ an open box containing $a$.  Since $J_p(a) \neq 0$, there exists some minor of the matrix $( \partial p_i / \partial x_j )_{i,j}$ has non-zero determinant.  By swapping variables and functions, we may assume that $J_{p^*}(\pi(a)) \neq 0$, where $\pi : F^n \rightarrow F^{n-1}$ is the projection onto the first $n-1$ coordinates and $p^*(y) = \pi(p(y, a_n))$.  By Lemma \ref{Lem_58dpmin}, there exists an open box $W \subseteq R^{n-1}$ with endpoints in $F$ containing $\pi(a)$ with $W \cap F^{n-1} \subseteq \pi(U)$ satisfying $p^* |_W$ is injective, $p^*(W)$ is open, and $(p^*)^{-1} |_{p^*(W) \cap F^{n-1}}$ is continuous.  The induction hypothesis says that $p^*(W \cap F^{n-1})$ has non-empty interior.  Hence, there exists $U_0 \subseteq (W \cap F^{n-1})$ an open box in the sense of $F$ so that $V_0 := p^*(U_0)$ is open (in $F^{n-1}$).  Then, $p^* |_{U_0}$ is a homeomorphism between $U_0$ and $V_0$.
 
 Notice that $U_0 \times I_n \subseteq U$ is an open box in the sense of $F$.  Define
 \begin{align*}
  U_1 = & \left\{ y \in U_0 : (\exists z \in F)\left( \{ p^*(y) \} \times (z, p_n(y, a_n)) \subseteq p(U_0 \times I_n) \right) \right\}. \\
  U_2 = & \left\{ y \in U_0 : (\exists z \in F) \left( \{ p^*(y) \} \times (p_n(y, a_n), z) \subseteq p(U_0 \times I_n) \right) \right\}. \\
  U_3 = & \ U_0 \setminus (U_1 \cup U_2).
 \end{align*}
 By Theorem \ref{Thm_GeneralizeDPMinSets} (1)${}_{n-1}$, we may assume that $U_0 = U_i$ for $i = 1$, $2$, or $3$.
 
 If $U_0 = U_1$, set $g(y)$ to the infimum of all $z$ witnessing the condition of $U_1$.  Using the functions $p_n(y, a_n)$ and $g(y)$ in Corollary \ref{Cor_BetweenTwoFunctions}, we get an open box $W \subseteq U_1$.  This is the desired conclusion.  A similar argument shows that, if $U_0 = U_2$, then $U_0$ has non-empty interior.
 
 So suppose $U_0 = U_3$.  Consider the definable set $Y = p_n(U_0 \times I_n)$.  Clearly $p_n(b, a_n) \in Y$ for all $b \in U_0$.  However, since $U_0 = U_3$, there are elements $z$ arbitrarily close to $p_n(b, a_n)$ for which $z \notin Y$.  By Lemma \ref{Lem_PierresLemma}, there is an open interval $I$ around $p_n(b, a_n)$ such that $I \cap Y = \{ p_n(b, a_n) \}$.  Define $h_1, h_2 : V_0 \rightarrow \overline{F}$ to be such that $(h_1(p^*(b)), h_2(p^*(b)))$ is the largest convex set witnessing this.  By Corollary \ref{Cor_BetweenTwoFunctions}, there is an open box
 \[
  W \subseteq \{ \langle p^*(y), z \rangle : y \in U_0, h_1(p^*(y)) < z < h_2(p^*(y)) \}.
 \]
 By continuity, $p^{-1}(W) \cap F^n$ is open.  However, $p^{-1}(W) \cap [U_0 \times I_n] \subseteq U_0 \times \{ a_n \}$.  Hence, it is not open.  Contradiction.
\end{proof}

The following lemma is contained in the proof of Proposition 5.4 in \cite{mms} and works for any field.  The proofs of (1), (2), and (3) follow from the proof of Theorem 1 in \cite{mmv}.  See the proof of Proposition 5.4 in \cite{mms} for more details.

\begin{lem}\label{Lem_Prop54Lem}
 Let $F$ a field and $\alpha \notin F$ algebraic over $F$.  Let $\alpha = \alpha_1, ..., \alpha_n$ be the conjugates of $\alpha$ over $F$ and let
 \[
  g(x_1, ..., x_n,y) = \prod_{i=1}^n \left( y - \sum_{j = 0}^{n-1} \alpha^j_i x_j \right) = \sum_{j=0}^{n-1} G_j(x_1, ..., x_n) y^j + y^n.
 \]
 Then,
 \begin{enumerate}
  \item $G_j(x) \in F[x]$ for all $j$.
  \item For $a \in F^n$, if $a_j \neq 0$ for some $j$, then $g(a, y)$ has no roots in $F$.
  \item There is some $d \in F^n$ such that $J_G(d) \neq 0$ and $d_j \neq 0$ for some $j$.
 \end{enumerate}
\end{lem}

\begin{proof}[Proof of Proposition \ref{Prop_54dpmin}]
 Fix $F$ a dp-minimal ordered field and let $R$ be its real closure.  Suppose, by means of contradiction, that there exists $\alpha \in (R \setminus F)$ arbitrarily close to $F$.  Construct $G = \langle G_1, ..., G_n \rangle$ as in Lemma \ref{Lem_Prop54Lem} above for this choice of $\alpha$, so conditions (1) through (3) hold (say (3) holds for some $d \in F^n$).  By Proposition \ref{Prop_59dpmin}, there exists an open $U \subseteq F^n$ with $d \in U$ so that $V := G(U)$ has non-empty interior.  By choosing $U$ sufficiently small, we may assume that, for all $c \in U$, $J_G(c) \neq 0$ and $c_j \neq 0$ for some $j$.  By Proposition \ref{Lem_58dpmin}, we may assume that $G |_U$ is a homeomorphism from $U$ to $V$.  Take $B \subseteq V$ an open box and, without loss of generality, suppose $e := G(d) \in B$.  Let $f : (R \setminus \{ 0 \}) \rightarrow R$ be the function
 \[
  f(y) = - y^n / y^{n-1} - \sum_{i=0}^{n-2} e_i y^i / y^{n-1}
 \]
 and let $h(y) = \sum_{i=0}^{n-1} d_i y^i$.  Note that $h(\alpha) \neq 0$ since $h(\alpha)$ is a root of $g(d,y)$, which has no roots in $F$ by Lemma \ref{Lem_Prop54Lem} (2).  It is not hard to show that $f(h(\alpha)) = e_{n-1}$ (see the proof of Proposition 5.4 in \cite{mms} for more details).  Hence, as $b \rightarrow \alpha$, $f(h(b)) \rightarrow e_{n-1}$.  Therefore, there is $b \in F$ so that $h(b) \neq 0$ and
 \[
  \langle e_0, e_1, ..., e_{n-2}, f(h(b)) \rangle \in B \subseteq V.
 \]
 Since $G |_U$ is a homeomorphism, there exists $c \in U$ so that $G(c) = \langle e_0, e_1, ..., e_{n-2}, f(h(b)) \rangle$.  So $c_j \neq 0$ for some $j$, hence by Lemma \ref{Lem_Prop54Lem} (2), $g(c,y)$ has no roots in $F$.  However, clearly $h(b)$ is a root of $g(c,y)$ in $F$, a contradiction.
\end{proof}

Here is where we must part ways with dp-minimality and impose the stronger condition of dp-smallness.  The main obstruction of using dp-minimality here is that dp-minimal ordered groups need not be divisible (for example, $(\mathbb{Z}, +, <)$).  For the remainder of this section, suppose that $\mathfrak{F} = (F; +, \cdot, <)$ is a dp-small ordered field.  Put on $F$ the archimedean valuation $v : F \rightarrow \Gamma$, where $v(a) \ge 0$ for $a \in F$ if and only if there exists $n \in \mathbb{N}$ such that $|a| < n$.  Let $\overline{F}$ be the residue field, $V$ the valuation ring (i.e., the convex hull of the prime field), and $M$ the maximal ideal in $V$ (i.e., the set of infinitesimals near zero).  So $\overline{F} = V / M$ is archimedean.  We now follow the remainder of the proof in \cite{mms}.

\begin{lem}\label{Lem_510co}
 The value group $\Gamma$ is divisible.
\end{lem}

\begin{proof}
 By Theorem \ref{thm_MainResults2} (2), $(F_+; \cdot, <)$ is divisible.  Hence, $\Gamma$ is also divisible.
\end{proof}

\begin{lem}\label{Lem_511co}
 The residue field $\overline{F}$ is real closed.
\end{lem}

\begin{proof}
 Suppose $\overline{F}$ is not real closed and let $p(x) \in V[x]$ be such that $\overline{p}(x)$ changes sign in $\overline{F}$ but has no root in $\overline{F}$.  Then, define
 \[
  M^* = \{ a \in F : (\forall b \in F)( |p(b)| > |a| ) \}.
 \]
 It is not hard to show that $M^* = M$, the set of infinitesimals of $F$ (for more details, see the proof of Proposition 5.11 of \cite{mms}).  Therefore, $v$ is definable.  By Lemma \ref{Lem_SIwCO}, $(\overline{F}; +, \cdot, 0, 1)$ is dp-small (and, in particular, dp-minimal).  Therefore, since $\overline{F}$ is archimedean, by Proposition \ref{Prop_54dpmin}, $\overline{F}$ is real closed.
\end{proof}

\begin{lem}\label{Lem_512co}
 The valued field $(F, v, \Gamma)$ is Henselian.
\end{lem}

\begin{proof}
 Suppose not.  Then there exists a polynomial
 \[
  p(x) = x^n + ax^{n-1} + \sum_{i=0}^{n-2} c_i x^i \in V[x]
 \]
 with $v(a) = 0$ and $v(c_i) > 0$ for $i < n-1$ that has no root in $F$.  Moreover, there is $\alpha \in R$ with $p(\alpha) = 0$, $v(a - \alpha) > 0$, and $v(p'(\alpha)) = 0$ (for details, see Theorem 4 of \cite{riben}).  Let
 \[
  S := \{ v(b - \alpha) : b \in F, v(b - \alpha) > 0 \},
 \]
 let $S^*$ be the convex subgroup of $\Gamma$ generated by $S$, and let $I := \{ b \in F : v(b) > S^* \}$.
 
 First, we show that $S$ is cofinal in $S^*$.  To see this, take $v(b-\alpha) \in S$ and let $c = b - p(b)/p'(b)$ ($v(b-\alpha) > 0$, so $v(p'(b)) = v(p'(\alpha)) = 0$, so $p'(b) \neq 0$).  Then, it is easy to check that $v(c-\alpha) \ge 2 v(b-\alpha)$ using Taylor's Theorem.  This shows that $S$ is cofinal in $S^*$.  Moreover, $I$ is definable in $\mathfrak{F}$.  To see this, notice that $v(a - \alpha) > 0$, hence $v(p(a)) = v(a - \alpha)$ by Taylor's Theorem (since $v(p'(\alpha)) = 0$).  By Lemma \ref{Lem_510co}, $\Gamma$ is divisible, hence there exists $d \in F$ so that $v(d) = v(p(a)) / 2 = v(a - \alpha) / 2$.  If $\alpha > a$, set $c = a + d$ and otherwise set $c = a - d$.  Therefore, $\alpha$ lies between $a$ and $c$ and, for all $b \in F$ between $a$ and $c$, $v(b - \alpha) > 0$.  Set $J$ to be the interval in $F$ between $a$ and $c$.  Hence, $I = \{ d \in F : (\forall b \in J)(|p(b)| > |d|) \}$.

 The valuation $v_0 : F \rightarrow \Gamma / S^*$ is definable, so look at the residue field $F_0$.  As in the proof of Lemma \ref{Lem_511co}, $\overline{\alpha} \in F_0$ (the image of $\alpha$ in $R_0$, the residue field of $R$ with respect to $v_0$).  Moreover, $\overline{\alpha}$ is in the convex hull of $\overline{J}$ in $R_0$, hence $\overline{\alpha} \in \overline{J}$.  However, for all $b \in J$, $v(p(b)) = v(b - \alpha) \in S^*$.  Therefore, $\overline{p}$ has no root in $\overline{J}$.  We see that $\overline{\alpha}$ directly contradicts this fact.
\end{proof}

\begin{proof}[Proof of Theorem \ref{Thm_CoRC}]
 Notice that (4) $\Rightarrow$ (1) $\Rightarrow$ (2) $\Rightarrow$ (3) is trivial.  So we need only show (3) $\Rightarrow$ (4).  Suppose $\mathfrak{F} = (F; +, \cdot, <)$ is a dp-small ordered field.  By Lemma \ref{Lem_510co}, Lemma \ref{Lem_511co}, and Lemma \ref{Lem_512co}, $(F, v, \Gamma)$ is a Henselian valued field with a divisible value group and a real closed residue field.  This implies that $F$ itself is real closed.
\end{proof}

\section{VC-Minimal Fields}\label{Sec_VCMinF}

In this section we move away from ordered fields and discuss VC-minimal fields in general.  We exhibit what is known about stable VC-minimal fields and state a conjecture about the nature of VC-minimal fields in general.

\begin{rem}\label{Rem_ACFVCMin}
 The theory of algebraically closed fields is VC-minimal.  In particular, it is the reduct of ACVF, which is certainly VC-minimal.
\end{rem}

In fact, piecing together results from \cite{kp} and \cite{ou}, we get a much stronger result.

\begin{thm}\label{Thm_dpStable}
 Let $\mathfrak{F} = (F; +, \cdot)$ be a field, let $T = \Th(\mathfrak{F})$, and suppose $T$ is stable and dp-minimal.  Then, $\mathfrak{F}$ is algebraically closed.
\end{thm}

\begin{proof}
 By Theorem 3.5 (iii) of \cite{ou}, any theory $T$ that is stable and dp-minimal has weight $1$.  By Corollary 2.4 of \cite{kp}, stable and strongly dependent (in particular, dp-minimal) fields with weight $1$ are algebraically closed.
\end{proof}

In particular, the theory of separably closed fields with positive Er\v{s}ov-invariant is not dp-minimal, or even strongly dependent.  This is not surprising, since this theory is not even superstable \cite{Wood}.

We get, as an immediate corollary, the following fact about Henselian valued fields.

\begin{cor}\label{Cor_ACFVClassify}
 Suppose that $\mathfrak{F} = (F; +, \cdot, |)$ is a dp-small Henselian valued field (with non-trivial valuation) with a stable residue field.  Then $\mathfrak{F}$ is an algebraically closed valued field.
\end{cor}

We have, in particular, that all stable VC-minimal fields are algebraically closed.  On the other hand, by Theorem \ref{Thm_CoRC}, all VC-minimal ordered fields are real closed.  We also know that all VC-minimal unstable theories interpret an infinite linear order (Theorem 3.5 of \cite{gl}).  This gives evidence for the following conjecture.

\begin{conj}[VC-minimal fields conjecture]\label{Conj_VCFields}
 Let $\mathfrak{F} = (F; +, \cdot)$ be a field and let $T = \Th(\mathfrak{F})$.  Then $T$ is VC-minimal if and only if $\mathfrak{F}$ is real closed or algebraically closed.
\end{conj}

This would, in turn, give a nice characterization of VC-minimal Henselian valued fields.

The gap in proving this conjecture is going from an infinite interpretable linear order (on a definable subset of $F^n$ for $n$ possibly much larger than $1$), to a field ordering on $F$.

\section*{Acknowledgements}

The author thanks Kobi Peterzil, Pierre Simon, and Charles Steinhorn for listening to various forms of the VC-minimal ordered fields argument and giving helpful suggestions.  We also thank Alfred Dolich, Joseph Flenner, Cameron Hill, and Sergei Starchenko for ideas around dp-smallness.

\end{document}